\documentclass[a4paper,9pt]{article}
\usepackage{amsfonts}
\usepackage{amssymb,amsthm,color,comment,graphicx,tikz}
\usepackage{amsmath}
\usepackage[normalem]{ulem}
\usepackage{esint}

\newcommand{\ds}{\displaystyle}

\newcommand{\R}{\mathbb{R}}

\newtheorem{theorem}{Theorem}
\newtheorem{lemma}[theorem]{Lemma}

\newtheorem{remark} {Remark}

\def\cp{\mathrm{cap}\,}

\newcommand{\Om}{\Omega}

\newcommand{\vps}{\varepsilon}
\newcommand{\bp}{\begin{proof}}
\newcommand{\ep}{\end{proof}}

\begin{document}
\title{On functionals involving the $p$-capacity and the $q$-torsional rigidity}

\author{{Michiel van den Berg} \\
School of Mathematics, University of Bristol\\
Fry Building, Woodland Road\\
Bristol BS8 1UG\\
United Kingdom\\
\texttt{mamvdb@bristol.ac.uk}\\
\\
{Nunzia Gavitone}\\
Dipartimento di Matematica e Applicazioni ``Renato Caccioppoli'' \\
Universit\`a degli Studi di Napoli Federico II\\
Via Cintia, Monte S. Angelo, I-80126 Napoli\\
Italy\\
\texttt{nunzia.gavitone@unina.it}}

\date{24 July 2025}\maketitle
\vskip 1.0truecm \indent
\begin{abstract}\noindent
Upper bounds are obtained for the $p$-capacity of compact sets in $\R^d$,  with $d \ge 2$ and $1<p<d$. Upper and lower bounds are obtained for the product of $p$-capacity and powers of the $q$-torsional rigidity over the collection of all non-empty, open, bounded and convex sets in
$\R^d$ with either a perimeter constraint, or a measure constraint, or a combination of perimeter and measure constraints. For some range of parameters we identify the ball  as the unique (up to homotheties)  maximiser or minimiser respectively.
\end{abstract}
%\vskip 1.0truecm
\begin{center} \textbf{Keywords}: $p$-capacity, $q$-torsional rigidity, perimeter, measure.
\end{center}
\mbox{}

\section{Introduction\label{sec1}}

For a compact set $K\subset \R^d,\,d\ge 2$ we recall, for $1<p<d$, a definition of its $p$- capacity $\cp_p(K)$. See for example \cite{HKM} or, for an equivalent definition, \cite{EG}.
\begin{align}\label{e1}
\cp_p(K)=&\inf\Big\{\int_{\R^d}|D u|^p: u\ge {\bf 1}_K\,\cap\, (Du\in L^p(\R^d))\,\cap (u\in C^0(\R^d))\,\nonumber\\&\cap\, (u\, \textup{vanishing at}\, \infty)\Big\},
\end{align}
where  $\bf 1_{\cdot}$ denotes the indicator function, and where vanishing at $\infty$ means \newline $(\forall \varepsilon>0)[|\{u>\varepsilon\}|<\infty]$,  and where $|\cdot|$ denotes $d$-dimensional Lebesgue measure.

It follows from \eqref{e1} that for a homothety $tK,$
\begin{equation}\label{e2}
\cp_p(tK)=t^{d-p}\cp_p(K),\,t>0.
\end{equation}
Furthermore if $K_1$ and $K_2$ are compact sets in $\R^d$ then $K_1\subset K_2$ implies $\cp_p(K_1)\le \cp_p(K_2)$.

For an open set $\Omega\subset \R^d$ we recall, for $q>1$, its $q$-torsional rigidity, or $q$-torsion for short:
\begin{equation}\label{e3}
T_q(\Omega)=\Bigg(\sup_{\psi\in W_0^{1,q}(\Omega)\setminus\{0\}} \frac{\displaystyle\left(\int_\Omega|\psi| \right)^{q}}{\displaystyle\int_\Omega |D \psi|^q }\Bigg)^{1/(q-1)},
\end{equation}

See for example \cite{JDE}. By \eqref{e3} we see that the $q$-torsion has the following scaling property:
\begin{equation}\label{e4}
T_q(t\Omega)=t^{d+q'}T_q(\Omega), \, q'=\frac{q}{q-1},\,t>0.
\end{equation}

We denote the closure, boundary, perimeter, diameter and measure of a measurable set $A$ by $\overline{A}$, $\partial A$, $P(A)$, $\textup{diam}(A)$ and $|A|$ respectively. Moreover, we denote the inradius of a non-empty set $A\subset \R^d$ by $\rho(A)=\sup_{x\in A}\sup\{r>0:B_r(x)\subset A\}$. The $a$-dimensional Hausdorff measure is  denoted by $\mathcal{H}^a(A)$.

Inequalities between torsional rigidity, first Dirichlet eigenvalue etc. go back at least as far as \cite{PSZ}. For more recent contributions we refer to \cite{MaNa}, \cite{cri}, \cite{MBP}, \cite{BBP1} and \cite{BBP2}.

In this paper we consider maximisation and minimisation problems of the functional
\begin{equation}\label{e5}
G_{p,q,r,\alpha}(\Om)=\frac{\cp_p(\overline{\Om})T_q(\Om)^r}{|\Om|^{\alpha}P(\Om)^{\beta}},
\end{equation}
over the collection of non-empty, open, bounded and convex sets.
Here $\alpha\ge 0$, and $\beta\ge 0$ are parameters which satisfy
\begin{equation}\label{e6}
d\alpha+(d-1)\beta=d-p+(d+q')r.
\end{equation}
In case the perimeter  term is absent, that is $\beta=0$, we put
\begin{equation}\label{e7}
\alpha^*=\frac1d(d-p+(d+q')r).
\end{equation}

The special case $p=q=2$ has been considered in \cite{vdB1,MB,MM}.

By scaling of measure and perimeter, \eqref{e2}, \eqref{e4} and \eqref{e6} we conclude that $G_{p,q,r,\alpha}(t\Om)=G_{p,q,r,\alpha}(\Om),\,t>0$.
It was shown in  \cite[Theorem 3]{vdB1} that if $0\le \alpha\le \frac{2}{d}$, then
\begin{align*}%\label{e9}
\sup\{G_{2,2,1,\alpha}(\Omega): \Omega\,\, \textup{non-empty, open, bounded, convex in}&\, \R^d\}\nonumber\\&
=G_{2,2,1,\alpha}(B_1),
\end{align*}
where $B_1$ is an open ball with radius $1$ in $\R^d$. Any open ball is a maximiser of $G_{2,2,1,\alpha}$.

A key ingredient in the proof of the previous statement is the following isoperimetric upper bound \cite[(6)]{vdB1} for the Newtonian capacity for non-empty, compact and convex sets in $\R^d$:
\begin{equation}\label{e10}
\cp_2(K)\le \frac{d-2}{d}\frac{P(K)^2}{|K|},
\end{equation}
with equality for a closed ball.

The proof of \eqref{e10} uses an idea going back to \cite{CFG}. There upper bounds for $\cp_2(K)$ were obtained by restricting the class of test functions to the ones depending only the distance to the boundary of $K$.
We recall the following notation from \cite{vdB1}. For a non-empty compact set $K$ we denote for $t>0$ its closed $t$-neighbourhood by
\begin{equation}\label{e11}
K_t=\{x\in\R^d: d_K(x)\le t\},
\end{equation}
where
\begin{equation*}%\label{e12}
d_K(x)=\min\{|x-y|:\,y\in K\},x\in \R^d,
\end{equation*}
is the distance to $K$ function.

If $K$ is compact then $\R^d\setminus K$ is open and consists of a countable union of open sets. Since $K$ is bounded there is precisely one unbounded component of its complement, which is denoted by $U_K$. Let $\tilde{K}= \R^d\setminus U_K$. Let $A_K$ be the union of all bounded components of the complement of $K$. Then $A_K$ is open, and $K=\R^d\setminus(A_K\cup U_K)\subset (\R^d\setminus U_K):=\tilde{K}$. It is straightforward to show that $\tilde{K}=K\cup A_K$, and that $\cp_2(K)=\cp_2(\tilde{K})$.
The  bound in Theorem \ref{the1}(i) will be given in terms of the perimeter of the closed $t$-neighbourhood of $\tilde{K}$.

If $K$ is compact and $\partial K$ is $C^2$, oriented by an outward unit normal vector field, then we denote the mean curvature map by $H:\partial K\rightarrow \R$, and define its integral by
\begin{equation*}%\label{ee}
M(K)=\int_{\partial K}Hd\mathcal{H}^{d-1}.
\end{equation*}
In Theorem \ref{the1}(iv) we give an upper bound for $\cp_p(K)$ in terms of $M(K)$ and $P(K)$.

\begin{theorem}\label{the1} Let $K\ne \emptyset$, and let $1<p<d$.
\begin{enumerate}
\item[\textup{(i)}] If  $K$ is compact in $\R^d,\,d\ge 2$, and if $\int_{(0,\infty)}dt\, P(\tilde{K}_t)^{-1/(p-1)}<\infty$, then
\begin{equation*}%\label{e13}
\cp_p(K)\le \cp_p(\tilde{K})\le \Big(\int_{(0,\infty)}dt\, P(\tilde{K}_{t})^{-1/(p-1)}\Big)^{1-p}<\infty.
\end{equation*}
with equality if $K$ is a closed ball.
\item[\textup{(ii)}]If $K$ is compact in $\R^d,\,d\ge 2,$ and if
\begin{equation*}%\label{e14}
\lim_{s\downarrow 0}\int_{(s,\infty)}dt\, P(\tilde{K}_t)^{-1/(p-1)}=\infty,
\end{equation*}
then $\cp_p(K)=0$.
\item[\textup{(iii)}]If $K$ is compact in $\R^d,\,d\ge 2,$ then
\begin{equation}\label{e15}
\cp_p(K)\le \inf_{a>0}\frac{1}{a^p}|K_a|.
\end{equation}
\item[\textup{(iv)}] If $K$ is compact, convex, and if $\partial K$ is $C^2$, then
\begin{equation}\label{e16}
\cp_p(K)\le \Big(\frac{d-p}{p-1}\Big)^{p-1}\frac{M(K)^{p-1}}{P(K)^{p-2}},
\end{equation}
with equality if $K$ is a closed ball.
\end{enumerate}
\end{theorem}

Inequalities for the $p$-capacity of compact, convex sets with smooth boundary in terms of integrals over powers of the mean curvature have been obtained previously in for example \cite[Theorem 3.1]{xiao}. We see that (3.3) in that paper jibes with \eqref{e16} for $p=2$, taking the different normalisations of the $p$-capacity into account. For $2<p<d$ \eqref{e16} implies, by H\"older's inequality the first inequality in (3.3), and inequality (3.4). Moreover \eqref{e16} also holds for $1<p<2$ and $d=2$. For the anisotropic version of \cite[Theorem 3.1]{xiao} we refer to \cite[Theorem 3]{li}.

It follows from the Aleksandrov-Fenchel inequality \eqref{e39} below (for $k=2,j=1,i=0$) together with \eqref{e16}, that
\begin{equation}\label{e17}
\cp_p(K)\le \Big(\frac{d-p}{d(p-1)}\Big)^{p-1}\frac{P(K)^p}{|K|^{p-1}},
\end{equation}
with equality if $K$ is any closed ball.  It can be shown that \eqref{e17} holds without the $C^2$ assumption on $\partial K$.

It was shown in \cite[Theorem 2(iii)]{vdB1} that if $d=3$, then the functional $G_{2,2,1,\alpha^*}(\Om)$ is bounded from above on the collection of  open, bounded,  convex sets.
The occurrence of a logarithmic factor for elongated convex, compact sets in the upper bound the Newtonian capacity implies the existence of a maximiser. See \cite[p.354]{vdB1}. These logarithmic factors for $d=3$ are well known. See for example \cite[p.260]{IMcK}.
In Theorem \ref{the2} we show that they occur for other values of $p$ and $d$.
Define the isoperimetric  ratio of $K$ by
\begin{equation}\label{e18}
I(K)=\frac{\omega_dd^d|K|^{d-1}}{P(K)^d},
\end{equation}
where $\omega_d=|B_1|$. By the isoperimetric inequality $0<I(K)\le 1$ with equality if and only if $K$ is a ball modulo sets of measure $0$. $I(K)$ is a measure of asymmetry.
\begin{theorem}\label{the2} Let $K$ be a convex, compact set and with non-empty interior. If $p=d-1,\,d=3,4,...,$ then
\begin{equation*}%\label{e19}
\cp_{d-1}(K)\le d^{2-d}(d-1)2^{d-2}\frac{P(K)^{d-1}}{|K|^{d-2}}\Big(\log (I(K)^{-1})\Big)^{2-d}.
\end{equation*}
\end{theorem}

Theorem \ref{the2} enables us to prove that if  $p=q'r=d-1,\,d\ge 3$ then $G_{p,q,r,\alpha^*}(\Om)$ has a maximiser in the collection of non-empty, open, bounded, and convex sets in $\R^d$. See Theorem \ref{the_new1}(iv).

This paper is organised as follows. In Section \ref{sec2} below we prove Theorems \ref{the1}  and \ref{the2} respectively.
In Section \ref{sec3} we obtain results for the maximisation of $G_{p,q,r,\alpha^*}.$  We prove that for suitable parameters the ball is the unique
(up to hometheties) maximiser for $G_{p,q,r,\alpha}$ among the class of non-empty, open, convex sets in $\R^d$. In Section \ref{sec4} we obtain results for the minimisation of $G_{p,q,r,\alpha},$ and show that for some suitable range of parameters a minimiser exists.  In Section \ref{sec5} we make some final remarks.

\section{Proofs of Theorem \ref{the1} and Theorem \ref{the2}\label{sec2}}

\noindent{\it Proof of Theorem \ref{the1}.}
The starting point of the proof of Theorem \ref{the1} goes back to  \cite[Theorem 11]{CFG} where the authors obtain, for convex bodies $K$, an upper bound for $\cp(K)$ by restricting the test functions in \eqref{e1} to those depending on $d_K$ only. The strategy of the proof below follows the one of \cite[Theorem 1]{vdB1}.

Let $s>0$ be arbitrary, and let $\varphi=f(d_{\tilde{K}_s})$, where $f$  will be suitably chosen below. Then
\begin{equation*}%\label{e20}
|D\varphi|^p=(f'(d_{\tilde{K}_s}))^p|Dd_{\tilde{K}_s}|^p= (f'(d_{\tilde{K}_s}))^p.
\end{equation*}
By the coarea formula
\begin{equation}\label{e21}
\int_{\R^d}|D\varphi|^p= \int_{\R^d}(f'(d_{\tilde{K}_s}))^p=\int_{(0,\infty)}dr\, (f'(r))^pP(\tilde{K}_{s+r}).
\end{equation}
Minimising, formally, over all smooth $f$ with $f(0)=1$, $f(\infty)=0$ gives \newline $((f'(r))^{p-1}P(\tilde{K}_{s+r}))'=0.$ Hence $(f'(r))^{p-1}P(\tilde{K}_{s+r})=c$ for some $c\in \R$. It follows that
\begin{equation}\label{e22}
f(r)=c\int_{(r,\infty)}dt\,P(\tilde{K}_{s+t})^{-1/(p-1)}.
\end{equation}
In \cite{vdB1} it was shown that $t\mapsto P(\tilde{K}_t))^{-1}$ is continuous on $(0,\infty)$ except on a possibly countable, compact subset.
Since $K\ne \emptyset$ it contains a point say $0$. The $(s+t)$-neighbourhood of $0$ contains the closed ball $\overline{B(0;s+t)}$. This ball has perimeter $P(B_1)(s+t)^{d-1}$.
Since $\overline{B(0;s+t)}$ is convex and is a subset of $\tilde{K}_{s+t}$ we have
\begin{equation}\label{e23}
P(\tilde{K}_{s+t})\ge P(B_1)(s+t)^{d-1}.
\end{equation}
The integral in the right-hand side of \eqref{e22} converges since by \eqref{e23}
\begin{equation*}%\label{e24}
\int_{(r,\infty)}dt\,P(\tilde{K}_{s+t})^{-1/(p-1)}\le \frac{p-1}{d-p}P(B_1)^{-1/(p-1)}(s+r)^{(p-d)/(p-1)},
\end{equation*}
where we have used that $1<p<d$.
Since $f(0)=1$ we find that
\begin{equation}\label{e25}
f(r)=\frac{\int_{(r,\infty)}dt\, P(\tilde{K}_{s+t})^{-1/(p-1)}}{\int_{(0,\infty)}dt\, P(\tilde{K}_{s+t})^{-1/(p-1)}}.
\end{equation}
We now verify that $\varphi=f(d_{\tilde{K}_s})$ satisfies the constraints in \eqref{e1}.

\noindent(i) To prove continuity we have
\begin{align*}%\label{e26}
|\varphi(x)-\varphi(y)|&\le |f(d_{\tilde{K}_s}(x))-f(d_{\tilde{K}_s}(y))|\nonumber\\& \le \sup_{r\ge 0}|f'(r)||d_{\tilde{K}_s}(x)-d_{\tilde{K}_s}(y)|\nonumber\\&\le
\frac{\sup_{r\ge 0}P(\tilde{K}_{s+r})^{-1/(p-1)}}{\int_{(0,\infty)}dt\, P(\tilde{K}_{s+t})^{-1/(p-1)}}|x-y|\nonumber\\&
\le\frac{P(\tilde{K}_s)^{-1/(p-1)}}{\int_{(0,\infty)}dt\, P(\tilde{K}_{s+t})^{-1/(p-1)}}|x-y|,
\end{align*}
where we have used \eqref{e23}. Hence $\varphi$ is uniformly continuous. This in turn implies that $f\in L^1_\textup{loc}(\R^d)$.

\smallskip

\noindent(ii) To prove that $\varphi$ vanishes at infinity, we have by \eqref{e23} and \eqref{e25}
\begin{equation}\label{e27}
f(r)\le \frac{p-1}{d-p}P(B_1)^{-1/(p-1)}r^{(p-d)/(p-1)}\Big(\int_{(0,\infty)}dt\, P(\tilde{K}_{s+t})^{-1/(p-1)}\Big)^{-1}.
\end{equation}
By \eqref{e27}
\begin{align}\label{e28}
\{\varphi>\varepsilon\}&\subset\Big\{x\in\R^d: d_{\tilde{K}_s}(x)\le \Big(\frac{p-1}{d-p}\Big)^{(p-1)/(d-p)}\nonumber\\&
\times P(B_1)^{1/(p-d)}\Big(\int_{(0,\infty)}dt\, P(\tilde{K}_{s+t})^{-1/(p-1)}\varepsilon\Big)^{(p-1)/(p-d)}\Big\}.
\end{align}
Since $\tilde{K}_s$ is contained in a ball with radius $\textup{diam}(\tilde{K}_s),$ we have by \eqref{e28} that the level set $\{\varphi>\varepsilon\}$
is contained in a ball with radius
\begin{align*}%\label{e29}
\textup{diam}(\tilde{K}_s)&+\Big(\frac{p-1}{d-p}\Big)^{(p-1)/(d-p)}\nonumber\\&
\times P(B_1)^{1/(p-d)}\Big(\int_{(0,\infty)}dt\, P(\tilde{K}_{s+t})^{-1/(p-1)}\varepsilon\Big)^{(p-1)/(p-d)}.
\end{align*}
Hence this level set has finite measure.

\smallskip

\noindent(iii) To see that $D\varphi\in L^p(\R^d)$ we compute
by \eqref{e21} and \eqref{e25} that
\begin{equation}\label{e30}
\int_{\R^d}|D\varphi|^p\le\Big(\int_{(0,\infty)}dt\, P(\tilde{K}_{s+t})^{-1/(p-1)}\Big)^{1-p}<\infty.
\end{equation}
We conclude by (i), (ii) and (iii) above that $\varphi$ satisfies all constraints in \eqref{e1}. By \eqref{e1} and \eqref{e30},
\begin{equation}\label{e31}
\cp_p(K)= \cp_p(\tilde{K})\le \cp_p(\tilde{K}_s)\le \Big(\int_{(0,\infty)}dt\, P(\tilde{K}_{s+t})^{-1/(p-1)}\Big)^{1-p}<\infty.
\end{equation}
Theorem \ref{the1}(i) follows since $s>0$ was arbitrary.

\smallskip

Theorem \ref{the1}(ii) follows immediately from \eqref{e31}.

\smallskip

To prove Theorem \ref{the1}(iii) we let $a>0$ be arbitrary. Since $P(\tilde{K}_t)\le P(K_t),$ we have by \eqref{e11} that
\begin{equation}\label{e32}
\cp_p(K)\le \Big(\int_{(0,\infty)}dt\, P(K_{t})^{-1/(p-1)}\Big)^{1-p}.
\end{equation}
By H\"older's inequality with conjugate exponents $r=p$, $s=p/(p-1)$, and \eqref{e32} we have that
\begin{align*}%\label{e33}
a=\int_{[0,a]} dt&\le \Big(\int_{[0,a]}dt\,P(K_t)^{-1/p}(P(K_t))^{1/p}\nonumber\\&\le
\Big(\int_{[0,\infty)}dt\,P(K_t)^{-1/(p-1)}\Big)^{(p-1)/p}|K_a|^{1/p}\nonumber\\&
\le \cp_p(K)^{-1/p}|K_a|^{1/p}.
\end{align*}
This implies \eqref{e15} since $a>0$ was arbitrary.

\smallskip

To prove Theorem \ref{the1}(iv) we recall Steiner's formula for compact convex $K$ with $C^2$ boundary:
\begin{equation*}%\label{e34}
|K_t|=\sum_{n=0}^d \binom{d}{n} W_n(K)t^n,\,t>0,
\end{equation*}
where the $W_n(K)$ are the quermassintegrals for $K$. See \cite[Chapter 4(4.1)]{RS}.
These quermassintegrals can be expressed in terms of integrals over the surface $\partial K$ of polynomials in the $d-1$ principal curvatures.
In particular
\begin{equation}\label{e35}
W_0(K)=|K|, \ W_1(K)=d^{-1}P(K),\, W_2(K)=d^{-1}\int_{\partial K} Hd\mathcal{H}^{d-1}, W_d(K)=\omega_d.
\end{equation}
The coarea formula gives
\begin{equation}\label{e36}
P(K_t)=\sum_{n=1}^d n\binom{d}{n} W_n(K)t^{n-1},\,t>0.
\end{equation}
By the change of variable
\begin{equation}\label{e37}
t=\frac{W_1(K)}{W_2(K)}\theta,
\end{equation}
we obtain by \eqref{e32}, \eqref{e36} and \eqref{e37},
\begin{align}\label{e38}
&\int_{(0,\infty)}dt\, P(K_t)^{-1/(p-1)}\nonumber\\& =\frac{W_1(K)^{(p-2)/(p-1)}}{W_2(K)}\int_{(0,\infty)}d\theta \Big(\sum_{n=1}^d n\binom{d}{n} \frac{W_n(K)W_1(K)^{n-2}}{W_2(K)^{n-1}}\theta^{n-1}\Big)^{-1/(p-1)}.
\end{align}
The Aleksandrov-Fenchel inequalities \cite[(7.66)]{RS} read
\begin{equation}\label{e39}
W_j(K)^{k-i}\ge W_i(K)^{k-j}W_k(K)^{j-i},\,\,0\le i<j<k\le d.
\end{equation}
Let $j=2,\,i=1,\,k=n$ in \eqref{e39}. This gives
\begin{equation}\label{e40}
W_n(K)W_1(K)^{n-2}\le W_2(K)^{n-1}.
\end{equation}
By \eqref{e38} and \eqref{e40},
\begin{align*}%\label{e41}
\int_{(0,\infty)}dt\, (P(K_t))&^{-1/(p-1)}\nonumber\\&
\ge\frac{W_1(K)^{(p-2)/(p-1)}}{W_2(K)}\int_{(0,\infty)}d\theta \Big(\sum_{n=1}^d n\binom{d}{n} \theta^{n-1}\Big)^{-1/(p-1)}\nonumber\\&
=\frac{W_1(K)^{(p-2)/(p-1)}}{W_2(K)}d^{-1/(p-1)}\int_{(0,\infty)}d\theta \,(\theta+1)^{-(d-1)/(p-1)}\nonumber\\&
=\frac{W_1(K)^{(p-2)/(p-1)}}{W_2(K)}d^{-1/(p-1)}\frac{p-1}{d-p}\nonumber\\&
=\frac{p-1}{d-p}\frac{P(K)^{(p-2)/(p-1)}}{\int_{\partial K} Hd\mathcal{H}^{d-1}},
\end{align*}
where we have used \eqref{e35} in the last equality. This proves Theorem \ref{the1}(iv) by \eqref{e32}.
\hspace*{\fill }$\square $

\medskip

\noindent{\it Proof of Theorem \ref{the2}.} By \eqref{e32}, \eqref{e36} and the change of variable
\begin{equation*}% \label{e42}
t=\frac{W_0(K)}{W_1(K)}\theta
\end{equation*}
we obtain that
\begin{align}\label{e43}
&\int_{(0,\infty)}dt\, P(K_t)^{-1/(p-1)}\nonumber\\&=\frac{W_0(K)}{W_1(K)^{p/(p-1)}}\int_{(0,\infty)}d\theta \Big(\sum_{n=1}^d n\binom{d}{n} \frac{W_n(K)W_0(K)^{n-1}}{W_1(K)^n}\theta^{n-1}\Big)^{-1/(p-1)}.
\end{align}
Applying \eqref{e39} with $j=1,\,i=0,\,k=n$, and using that $p=d-1$ gives by \eqref{e43} that
\begin{align}\label{e44}
\int_{(0,\infty)}dt\, P(K_t)^{-1/(d-2)}&\ge\frac{W_0(K)}{W_1(K)^{(d-1)/(d-2)}d^{1/(d-2)}}\nonumber\\&\times\int_{(0,\infty)}d\theta \Big((1+\theta)^{d-1}-\theta^{d-1}+I(K)\theta^{d-1}\Big)^{-1/(d-2)}.
\end{align}
To bound the integral in the right-hand side of \eqref{e44} from below we use
\begin{equation*}%\label{e45}
(1+\theta)^{d-1}-\theta^{d-1}\le (d-1)(1+\theta)^{d-2},\,\theta\ge 0,
\end{equation*}
and
\begin{equation*}%\label{e46}
I(K)\theta^{d-1}\le (d-1)I(K)(1+\theta)^{d-1},\,\theta\ge 0.
\end{equation*}
This gives, by \eqref{e44} and the inequality $(x+y)^{\alpha}\le x^{\alpha}+y^{\alpha},\,x\ge 0,\,y\ge 0,\newline\,0<\alpha\le 1$, that
\begin{align}\label{e47}
\int_{(0,\infty)}dt\, P(K_t)&^{-1/(d-2)}\ge\frac{W_0(K)}{W_1(K)^{(d-1)/(d-2)}(d(d-1))^{1/(d-2)}}\nonumber\\&\,\,\times\int_{(0,\infty)}d\theta \Big((1+\theta)^{d-2}+I(K)(1+\theta)^{d-1}\Big)^{-1/(d-2)}\nonumber\\&
\ge \frac{W_0(K)}{W_1(K)^{(d-1)/(d-2)}(d(d-1))^{1/(d-2)}}\nonumber\\&\,\,\times\int_{(0,\infty)}d\theta \Big(1+\theta+I(K)^{1/(d-2)}(1+\theta)^{(d-1)/(d-2)}\Big)^{-1}.
\end{align}
We restrict the interval of integration to those values of $\theta$ for which \newline $I(K)^{1/(d-2)}(1+\theta)^{(d-1)/(d-2)}\le 1+\theta$. That is  $0\le\theta\le \theta^*$, where
\begin{equation}\label{e48}
\theta^*=I(K)^{-1}-1.
\end{equation}
It follows that
\begin{align}\label{e49}
\int_{(0,\infty)}&d\theta \Big(1+\theta+I(K)(1+\theta)^{(d-1)/(d-2)}\Big)^{-1}\nonumber\\&\ge \int_{(0,\theta^*)}d\theta \Big(1+\theta+I(K)(1+\theta)^{(d-1)/(d-2)}\Big)^{-1}\nonumber\\&
\ge \frac12\int_{(0,\theta^*)}d\theta (1+\theta)^{-1}\nonumber\\&
=\frac{1}{2}\log (I(K)^{-1}).
\end{align}
Theorem \ref{the2} follows from \eqref{e32}, \eqref{e47}, \eqref{e48} and \eqref{e49}.
\hspace*{\fill }$\square $

\section{Maximisation of $G_{p,q,r,\alpha}(\Omega)$}\label{sec3}
In this section we obtain results for the  maximisation of the functional  $G_{p,q,r,\alpha}$, defined in \eqref{e5}, in the class of non-empty, open, bounded and convex sets in  $\R^d$.
This  type of problem has been studied in \cite{vdB1}, \cite{MB} and \cite{MM} in the special case
$p=q=2$,  $\beta=0$ and $\alpha=\alpha^*$, with $\alpha^*$ defined in \eqref{e7}.  Our first result concerns  the maximisation problem of  $G_{p,q,r,\alpha}(\Omega)$ in the class of non-empty, open, bounded and convex sets in $\R^d$.
\begin{theorem}\label{the_new1}
   Let $d=2,3,...,$ $1<p<d$, $q>1$ and let $\alpha\ge0,\,\beta \ge0$ as in \eqref{e6}. Let $C_2$ be defined by
\begin{equation}
\label{c_2}
C_2=\begin{cases}
\frac{1}{d^{r(q'-1)}(d+q')^r},\,\, \, r\le0,\\
\frac{d^{q'r}}{(q'+1)^r}\hspace{11mm},\,\, r>0.
\end{cases}
\end{equation}

Then the following statements hold:
\begin{itemize}
\item[\textup{(i)}] If
\begin{equation*}%\label{e50}
q'r\ge \begin{cases}p-\beta,&1<p<  d-1,\\
&\\
(d-1)(d-p)-\beta,& d-1\le p<d,
\end{cases}
\end{equation*}
then
\begin{align}\label{e51}
&\sup \{G_{p,q,r,\alpha}(\Om):\,\Om\, \textup{non-empty, open,  bounded, convex in}\,\R^d\ \}\nonumber\\&\le \begin{cases}C_2\left(d+q'\right)^rd^{(2-\frac1d)(q'r+\beta-p)-r}G_{p,q,r,\alpha}(B_1),\,1<p< d-1,\\ \\
C_2(d+q')^rd^{2q'r+2\beta+d-p-r-\frac{1}{d}(q'r+\beta+d-p)}G_{p,q,r,\alpha}(B_1),
 d-1\le p<d.
\end{cases}
\end{align}
\item[\textup{(ii)}]If
\begin{equation}\label{e52}
q'r> \begin{cases}p-\beta &,1<p< d
-1,\\
&\\
(d-1)(d-p)-\beta&, d-1\le p<d,
\end{cases}
\end{equation}
then the variational problem in the left-hand side of \eqref{e51} has a maximiser. Any such  maximiser, denoted by $\Omega^*$, satisfies
\begin{align}\label{e53}
&\frac{2\rho(\Omega^*)}{\textup{diam}(\Omega^*)}\nonumber\\&\ge \begin{cases} d^{1-2d+\frac{dr}{q'r+\beta-p}}\big(C_2(q'+d)^r\big)^{-\frac{d}{q'r-p+\beta}},\hspace{16mm},1<p< d-1,\\ \\
d^{1+\frac{d(r-2q'r-2\beta)}{q'r+\beta-(d-p)(d-1)}}\big(C_2(d+q')^r\big)^{-\frac{d}{q'r+\beta-(d-p)(d-1)}},\,d-1\le p<d.
\end{cases}
\end{align}
\item[\textup{(iii)}] If $r>0$ and $\beta \ge p$, then
\begin{align*}%\label{e81}
\sup\{G_{p,q,r,\alpha}(\Om): \Om\, &\textup{non-empty, open, bounded, and convex in}\,\R^d\}\nonumber\\&\hspace{40mm}=G_{p,q,r,\alpha}(B_1),
\end{align*}
and the maximiser $B_1$ is unique (up to homotheties).
\\
\item[\textup{(iv)}]  If $q'r=p=d-1,\,\alpha=\alpha^*$, and if $d=3,4,...$, then the variational problem in the left-hand side of \eqref{e51} has a maximiser. Any such maximiser, denoted by $\Omega^*$
satisfies
\begin{equation}\label{e54}
\frac{2\rho(\Omega^*)}{\textup{diam}(\Omega^*)}\ge  d^{1-2d}e^{-2(d-2)d^{\frac{d-1-r}{d-2}}(d-1)^{\frac{1}{d-2}}\big(\frac{q'+d}{q'+1}\big)^{\frac{r}{d-2}}}.
\end{equation}
\end{itemize}
\end{theorem}
Statement (iii) above asserts that if the power of the perimeter term in the denominator is not too small, then a ball is the unique maximiser in the class of non-empty, open, bounded,  convex sets in $\R^d$. In the special case where $p=2, q=2, r=1$ and $d\ge 3$ we obtain  with \eqref{e6} $\alpha\le \frac2d$. This jibes with \cite[Theorem 3]{vdB1}.

In order to give the proofs of both Theorem  \ref{the_new1} and Theorem  \ref{the_min} (in Section \ref{sec4}) we have the following bounds for the $q$-torsion.
\begin{lemma}\label{lem0}
If $\Omega\subset\R^d$ is an open, non-empty and convex set with finite measure, and if $q'>1,r\in\R$, then
\begin{equation}\label{e54a}
T_q(\Om)^r\le C_2\frac{|\Om|^{(1+q')r}}{P(\Om)^{q'r}},
\end{equation}
where $C_2$ is given by \eqref{c_2},
and
\begin{equation}\label{e54b}
T_q(\Om)^r\ge C_1\frac{|\Om|^{(1+q')r}}{P(\Om)^{q'r}},
\end{equation}
where $C_1$ is given by \eqref{c1}.
\end{lemma}
\begin{proof}
It was shown in \cite{JDE} that
\begin{equation}\label{e56}
 \displaystyle \frac{1}{d^{q'-1}(d+q')}\le\frac{T_q(\Omega)}{|\Omega|\rho(\Omega)^{q'}}\le \frac{1}{q'+1},
\end{equation}
where the equality in the upper bound is asymptotically sharp for  a sequence of thinning cuboids. The equality in the lower bound holds for balls.

It was shown in \cite{Anona} that if
 $\Omega\subset \R^d$ is  a non-empty, open, bounded, and convex set, then
\begin{equation}
\label{e59}
\frac{1}{\rho(\Omega)} \le \displaystyle \frac{P(\Omega)}{|\Omega|} \le \displaystyle \frac{d}{\rho(\Omega)},
\end{equation}
with equality in the right-hand side holds if and only if the inner parallel set of $\Om$ is a homothety of $\Om$ with respect to the centre of the inball. The left-hand side is asymptotically sharp for thinning cuboids.
The upper bound in \eqref{e54a} follows for $r>0$ by the upper bounds in \eqref{e56} and \eqref{e59}. The upper bound in \eqref{e54a} follows for $r\le0$ by the lower bound in \eqref{e56} and the lower bound in \eqref{e59}.

The lower bound in \eqref{e54b} for $r>0$ follows from the lower bound in \eqref{e56} and the first inequality in \eqref{e59}. The lower bound for $r<0$ follows from  the upper bound in \eqref{e56} and the second inequality in \eqref{e59}.
\end{proof}

\smallskip

\begin{lemma}\label{lem1} If $\Omega$ is a non-empty, open, bounded, and convex set in $\R^d,\,d\ge 2$, then
\begin{equation}\label{e61a}
\frac{|\Om|^{(d-1)/d}}{P(\Om)}\le \frac{d^{(d-1)/d}}{\omega^{1/d}_d}\Big(\frac{2\rho(\Om)}{\textup{diam}(\Om)}\Big)^{1/d},
\end{equation}
\begin{equation}\label{e61d}
\textup{diam}(\Om)\le \frac{2d^{d-1}}{\omega_d}\frac{P(\Om)^{d-1}}{|\Om|^{d-2}},
\end{equation}
and
\begin{equation}\label{e61aa}
\frac{|\Om|^{(d-1)/d}}{P(\Om)}\ge \frac{1}{d\omega_d^{1/d}}\frac{2\rho(\Om)}{\textup{diam}(\Om)}.
\end{equation}
\end{lemma}
\begin{proof}
We have by the lower bound in \eqref{e59} that for $d\ge 2$,
\begin{equation}\label{e61b}
\frac{|\Om|}{P(\Om)^{d/(d-1)}}\le\frac{\rho(\Om)}{P(\Om)^{1/(d-1)}}.
\end{equation}
Let $E(a)$ be the John's ellipsoid for $\Omega$ with semi-axes $a_1,...,a_d$. Without loss of generality we may assume that $a_1 \ge a_2 \ge \cdots \ge a_d$, and that $0$ is the centre of $E(a)$.
Then $E(a)\subset \Omega\subset E(da)$. It follows that $a_d\le \rho(\Omega)\le da_d$, and that $2a_1\le \textup{diam}(\Omega)\le 2da_1$. By monotonicity of the perimeter under inclusion on the class of convex sets we have by the lower bound in \eqref{e59} that
\begin{align}\label{e61}
P(\Om)&\ge P(E(a))\ge \frac{|E(a)|}{\rho(E(a))}=\omega_d\Pi_{i=1}^{d-1}a_i\ge \omega_da_1a_d^{d-2}\nonumber\\&
\ge \frac{\omega_d}{2}\textup{diam}(E(a))\rho(E(a))^{d-2}\nonumber\\&
\ge \frac{\omega_d}{2d^{d-1}}\textup{diam}(\Om)\rho(\Om)^{d-2}.
\end{align}
By \eqref{e61b} and \eqref{e61},
\begin{equation*}%\label{e61c}
\frac{|\Om|}{P(\Om)^{d/(d-1)}}\le\frac{d}{\omega_d^{1/(d-1)}}\Big(\frac{2\rho(\Om)}{\textup{diam}(\Om)}\Big)^{1/(d-1)}.
\end{equation*}
This implies \eqref{e61a}.

To prove \eqref{e61d} we use \eqref{e61} to obtain that
\begin{equation}\label{e61e}
\textup{diam}(\Om)\le \frac{2d^{d-1}}{\omega_d}P(\Om)\rho(\Om)^{2-d},
\end{equation}
and use the lower bound for $\rho $ in \eqref{e59}. This yields \eqref{e61d}.

To prove \eqref{e61aa} we note that by \eqref{e59}
\begin{equation}\label{e61ab}
\frac{|\Om|}{P(\Om)}\ge \frac{\rho(\Om)}{d}.
\end{equation}
The isodiametric inequality \cite[p.69]{EG} reads
$$ |\Om|\le \omega_d\Big(\frac{\textup{diam}(\Om)}{2}\Big)^d.$$
It follows that
\begin{equation}\label{e61ac}
|\Om|^{-1/d}\ge 2\omega_d^{-1/d}\textup{diam}(\Om)^{-1}.
\end{equation}
Inequality \eqref{e61aa} follows by \eqref{e61ab} and \eqref{e61ac}.
\end{proof}

\smallskip

\begin{proof}[Proof of Theorem \ref{the_new1}]

To prove the first case under (i) we let $q'r\ge p-\beta$. By \eqref{e5}, \eqref{e17} and  \eqref{e54a} we get that
\begin{align*}%\label{e62}
G_{p,q,r,\alpha}(\Om)&\le \Big(\frac{d-p}{d(p-1)}\Big)^{p-1}C_2 \frac{P(\Om)^{p}}{|\Om|^{p-1}}\frac{|\Om|^{(1+q')r-\alpha}}{P(\Om)^{\beta+q'r}}\nonumber\\&
=C_2\Big(\frac{d-p}{d(p-1)}\Big)^{p-1}\Big(\frac{|\Om|^{(d-1)/d}}{P(\Om)}\Big)^{\beta+q'r-p},
\end{align*}
where we have used \eqref{e6} to eliminate $\alpha$. Since $\beta+q'r-p\ge 0$ we  may use \eqref{e61a} to obtain that
\begin{equation}\label{e63}
G_{p,q,r,\alpha}(\Om)\le C_2\,\Big(\frac{d-p}{d(p-1)}\Big)^{p-1}\Big(\frac{d^{d-1}}{\omega_d}\Big)^{\frac1d(q'r-p +\beta)}\Big(\frac{2\rho(\Om)}{\textup{diam}(\Om)}\Big)^{\frac1d(q'r-p +\beta)}.
\end{equation}

A straightforward computation using \eqref{e56} with $\Om=B_1$, and \eqref{e17} with $K=\overline{B_1}$ yields
\begin{equation}\label{e64}
 G_{p,q,r,\alpha}(B_1)= \Big(\frac{d-p}{p-1}\Big)^{p-1}  \frac{d^{1+r-q'r -\beta}}{(d+q')^r}\omega_d^{\frac1d(p-q'r -\beta )}.
\end{equation}
By \eqref{e63} and \eqref{e64},
\begin{equation}\label{e65}
G_{p,q,r,\alpha}(\Om)\le C_2\left(d+q'\right)^rd^{(2-\frac1d)(q'r+\beta-p)-r}\Big(\frac{2\rho(\Om)}{\textup{diam}(\Om)}\Big)^{\frac1d(q'r-p+\beta)}G_{p,q,r,\alpha}(B_1).
\end{equation}
This proves the first case under (i) since $2\rho(\Om)\le \textup{diam}(\Om)$.

\smallskip

To prove the second case under (i) we let  $q'r\ge (d-1)(d-p)-\beta$. We first note that $\Om$ is contained in a ball with radius $\textup{diam}(\Om)$.
Using \eqref{e17} once more with $K$ the closure of a ball with radius $\textup{diam}(\Om)$
we obtain that
\begin{equation}\label{e66}
\cp_p(\overline{\Om})\le \Big(\frac{d-p}{p-1}\Big)^{p-1}d\omega_d\textup{diam}(\Om)^{d-p}.
\end{equation}
By \eqref{e66} and  \eqref{e54a} we get
\begin{align*}%\label{e67}
G_{p,q,r,\alpha}(\Om)&\le \Big(\frac{d-p}{p-1}\Big)^{p-1}C_2\,d\omega_d\frac{\textup{diam}(\Om)^{d-p}|\Om|^{(1+q')r-\alpha}}{P(\Om)^{q'r+\beta}}\nonumber\\&
\le \Big(\frac{d-p}{p-1}\Big)^{p-1}C_2\,d\omega_d\Big(\frac{2d^{d-1}}{\omega_d}\Big)^{d-p}\Big(\frac{|\Om|^{(d-1)/d}}{P(\Om)}\Big)^{q'r+\beta-(d-1)(d-p)},
\end{align*}
where we have used \eqref{e61e} in the last inequality, and \eqref{e6} to eliminate $\alpha$.

Since $q'r\ge (d-1)(d-p)-\beta$ we use \eqref{e61a} to obtain that

\begin{align}\label{e68}
 G_{p,q,r,\alpha }(\Om)&\le \Big(\frac{d-p}{p-1}\Big)^{p-1}C_2\,d\omega_d\Big(\frac{2d^{d-1}}{\omega_d}\Big)^{d-p}\frac{d^{\frac{d-1}{d}(q'r+\beta-(d-1)(d-p))}}{\omega^{\frac1d(q'r+\beta-(d-1)(d-p))}_d}\nonumber\\&
 \hspace{4mm}\times\Big(\frac{2\rho(\Om)}{\textup{diam}(\Om)}\Big)^{\frac1d(q'r+\beta-(d-1)(d-p))}\nonumber\\&
 =\Big(\frac{d-p}{p-1}\Big)^{p-1}C_2\frac{d^{\frac{d-1}{d}(q'r+\beta+d-p)+1}}{\omega_d^{\frac1d(q'r+\beta-p)}}\nonumber\\&\hspace{4mm}\times\Big(\frac{2\rho(\Om)}{\textup{diam}(\Om)}\Big)^{\frac1d(q'r+\beta-(d-1)(d-p))}.
\end{align}

By \eqref{e64} and \eqref{e68}
\begin{align*}%\label{e68a}
 G_{p,q,r,\alpha}(\Om)&\le C_2(d+q')^rd^{2q'r+2\beta+d-p-r-\frac1d(q'r+\beta+d-p)}\nonumber\\&\times\Big(\frac{2\rho(\Om)}{\textup{diam}(\Om)}\Big)^{\frac1d(q'r+\beta-(d-1)(d-p))}G_{p,q,r,\alpha}(B_1).
\end{align*}
This proves the second case under (i).

Note that the requirements $1<p<d-1$ and $d-1\le p<d$ respectively were not needed in the proof of (i). They were just included to distinguish between the two cases under (i). The second case, where $d-1\le p<d$, gives a less restrictive lower bound for $q'r$. For $q'r\ge p-\beta$ both upper bounds in \eqref{e51} hold.

\smallskip

To prove (ii) we can proceed as in \cite{MB}. First note that the supremum is finite by (i). Let $(\Om_n)$ be a maximising sequence. Without loss of generality we may assume that  $G_{p,q,r,\alpha}(\Om_n)>G_{p,q,r,\alpha}(B_1)$ and that $(G_{p,q,r,\alpha}(\Om_n))$  is a strictly increasing sequence, for otherwise $B_1$ is a maximiser, and there is nothing to prove. Let $\Om$ be an element of this maximising sequence. By \eqref{e65} or \eqref{e68}
$\textup{diam}(\Om)/\rho(\Om)$ is uniformly bounded. Since $G_{p,q,r,\alpha}(\Om)$ is invariant under homotheties we may fix $\rho(\Om)=1$. Then $\textup{diam}(\Om)$ is uniformly bounded. By convexity of the elements $\Om_n$ there exists a subsequence of translates of $(\Om_n)$ which converges in the Hausdorff metric to say $\Om_{p,q,r,\alpha}$. By continuity of $p$-capacity, $q$-torsion, and measure we have that $\lim G_{p,q,r,\alpha}(\Om_n)=G_{p,q,r,\alpha}(\Om_{p,q,r,\alpha}).$
Hence $\Om_{p,q,r,\alpha}$ is a maximiser. Since $G_{p,q,r,\alpha}(\Om_{p,q,r,\alpha})>G_{p,q,r,\alpha}(B_1)$, \eqref{e65} and \eqref{e68} imply \eqref{e53}. If $\Om_{p,q,r,\alpha}=B_1$ then it also satisfies \eqref{e53}. This proves the assertion under (ii).

 To prove (iii) we let $\Omega^{\sharp} $ be  a ball with $|\Omega^{\sharp}|=|\Omega|$. Then  (see for instance \cite{ta})
\begin{equation}\label{e82}
T_q(\Omega) \le T_q(\Omega^{\sharp})=\frac{1}{d^{q'}\omega_d^{\frac{q'}{d}}}\frac{d}{q'+d}|\Omega|^{\frac{q'}{d}+1}
\end{equation}
with equality if $\Om$ is a ball. Rewriting \eqref{e82} gives
\begin{equation}\label{e83}
T_q(\Omega) \le T_q(B_1)\Big(\frac{|\Omega|}{|B_1|}\Big)^{\frac{q'}{d}+1}.
\end{equation}
Rewriting \eqref{e17} gives
\begin{equation}\label{e84}
\cp_p(K)\le \cp_p(\overline{B_1})\Big(\frac{P(K)}{P(B_1)}\Big)^p\Big(\frac{|B_1|}{|K|}\Big)^{p-1}.
\end{equation}
By \eqref{e83} and \eqref{e84} we have for $1<p<d,\,q>1$, $r>0$  and $\Om$ convex
\begin{align*}%\label{e85}
G_{p,q,r,\alpha}&(\Om)\nonumber \\
&\le T_q(B_1)^r\cp_p(\overline{B_1})\Big(\frac{|\Omega|}{|B_1|}\Big)^{\frac{q'r}{d}+r}\Big(\frac{P(\Om)}{P(B_1)}\Big)^p\Big(\frac{|B_1|}{|\Om|}\Big)^{p-1}|\Om|^{-\alpha}P(\Om)^{-\beta}\nonumber\\&
=G_{p,q,r,\alpha}(B_1)\Big(\frac{|\Om|}{|B_1|}\Big)^{(\beta-p)(d-1)/d}\Big(\frac{P(B_1)}{P(\Om)}\Big)^{p-\beta}.
\end{align*}
By the isoperimetric inequality we have for any $\gamma\ge 0$,
\begin{equation*}%\label{e86}
\frac{|\Om|^{\gamma/d}}{P(\Om)^{\gamma/(d-1)}}\le \frac{|B_1|^{\gamma/d}}{P(B_1)^{\gamma/(d-1)}}.
\end{equation*}
Let
\begin{equation*}%\label{e87}
\gamma=(\beta-p)(d-1)
\end{equation*}
Then $\gamma\ge 0$ is equivalent to $\beta \ge p$, and this completes the proof of (iii).

To prove (iv) we obtain by Lemma \ref{lem0}, Theorem \ref{the2} and \eqref{e7} that for $q'r=p=d-1$,
\begin{equation}\label{e69}
G_{d-1,q,r,\alpha^*}(\Om)\le d(d-1)2^{d-2}(q'+1)^{-r}\Big(\log(I(\Om)^{-1})\Big)^{2-d},
\end{equation}
where we have used in \eqref{c_2} that $r>0$. By \eqref{e64}
\begin{equation} \label{e70}
G_{p,q,r,\alpha^*}(B_1)=(d(d-2))^{2-d}d^r(q'+d)^{-r}.
\end{equation}
By \eqref{e69} and \eqref{e70},
\begin{align}\label{e71}
G_{d-1,q,r,\alpha^*}(\Om)&\le d^{d-1-r}(d-1)(d-2)^{d-2}2^{d-2}\Big(\frac{q'+d}{q'+1}\Big)^r\nonumber\\& \,\,\ \times G_{d-1,q,r,\alpha^*}(B_1)\Big(\log(I(\Om)^{-1})\Big)^{2-d}.
\end{align}
By \eqref{e18} and \eqref{e61a}
\begin{equation}\label{e72}
I(\Om)\le d^{2d-1}\frac{2\rho(\Om)}{\textup{diam}(\Om)}.
\end{equation}
We now argue as under (ii) to conclude that either $B_1$ is a maximiser of $G_{d-1,q,r,\alpha^*}(\Om)$ among the convex sets or there exists a maximising sequence $(\Om_n)$ with
$(G_{d-1,q,r,\alpha^*}(\Om_n))$ increasing with limit strictly larger than  $G_{d-1,q,r,\alpha^*}(B_1)$. As the functional is scaling invariant we may fix $\rho(\Om_n)=1$.
This gives, by \eqref{e71} and \eqref{e72}, that $\textup{diam}(\Om_n)$ is uniformly bounded in $n$. By convergence and continuity we conclude that a maximiser, say $\Om^*$, exists and satisfies
\begin{equation}\label{e73}
G_{d-1,q,r,\alpha^*}(\Om^*)>G_{d-1,q,r,\alpha^*}(B_1).
\end{equation}
By \eqref{e71} and \eqref{e73},
\begin{align*}%\label{ext}
\Big(\log(I(\Om)^{-1})\Big)^{d-2}\le d^{d-1-r}(d-1)(d-2)^{d-2}2^{d-2}\Big(\frac{q'+d}{q'+1}\Big)^r.
\end{align*}
This implies the assertion under \eqref{e54}.
\end{proof}

\section{Minimisation of $G_{p,q,r,\alpha}(\Omega)$}\label{sec4}

In this section we obtain some results for the minimisation problem of $G_{p,q,r,\alpha}$.
\begin{theorem}\label{the_min}    Let $d=2,3,...,$ $1<p<d$, $q>1$ and let $\alpha\ge0,\,\beta \ge0$ as in \eqref{e6}. Let $C_1$ be defined by
\begin{equation}
\label{c1}
C_1=\begin{cases}
\frac{1}{d^{r(q'-1)}(d+q')^r}, \, r>0,\\
 \frac{1}{(q'+1)^r}, \, r\le 0.
\end{cases}
\end{equation}
Then  the following statements hold:
\begin{itemize}
\item[\textup{(i)}]If $q'r>\displaystyle\frac{d-p}{d-1}-\beta$, then
$$\inf\{G_{p,q,r,\alpha}(\Om): \Om\,\textup{non-empty, open, bounded, convex in $\R^d$} \}=0.$$
\item[\textup{(ii)}] If $q'r\le\displaystyle \frac{d-p}{d-1}-\beta$, then
\begin{align}\label{e74}
&\inf\{G_{p,q,r,\alpha}(\Om): \Om\,\textup{non-empty, open, bounded convex in $\R^d$} \}\ge\nonumber\\& C_1\left( \frac{p-1}{p(d-1)}\right)^{p-1}d^{\frac{d(p-1)}{d-1}+\frac{p}{d}+\big(2-\frac1d\big)(q'r+\beta)-r-2}(d+q')^r G_{p,q,r,\alpha}(B_1).
\end{align}
\item[\textup{(iii)}]If $q'r<\displaystyle \frac{d-p}{d-1}-\beta$,
then the variational problem in the left-hand side of \eqref{e74} has a minimiser. Any such  minimiser, denoted by $\Omega_*$, satisfies
\begin{align*}
\left(\frac{2\rho(\Om_*)}{\textup{diam}(\Om_*)}\right)&^{\frac{d-p-(d-1)(q'r+\beta)}{d(d-1)}}\ge  C_1 \,\left( \frac{p-1}{p(d-1)}\right)^{p-1}(d+q')^r\nonumber\\&\times  d^{\frac{d(p-1)}{d-1}+\frac{p}{d}+\big(2-\frac1d\big)(q'r+\beta)-r-2}.
\end{align*}
 \item [\textup{(iv)}] If $q'r<\displaystyle \frac{d-p}{d-1}-\beta$ then
\begin{equation*}
\sup \{G_{p,q,r,\alpha}(\Om): \Om\, \textup{non-empty, open, convex in}\,\R^d \}=\infty.
\end{equation*}
 \item [\textup{(v)}] If $r<0$, then
\begin{align*}\inf &\{G_{p,q,r,\alpha^*}(\Om): \Om\,\textup{non-empty, open, bounded and convex in $\R^d$} \}\nonumber\\&=G_{p,q,r,\alpha^*}(B_1),
\end{align*}
and $B_1$ is a minimiser, unique up to homotheties.
\end{itemize}
\end{theorem}

\smallskip

Assertion (i) above jibes with \cite[Theorem 3(i)]{MB} in the special case $p=q=2,\,\beta=0$.

\begin{proof}[Proof of Theorem \ref{the_min}]
Let $0<\varepsilon<1$ and let $E_{c_\varepsilon}$ be the closed ellipsoid with $d-1$ semi-axes of length $\varepsilon^{-1/(d-1)},$ and one semi-axis of length $\varepsilon$.
Then $|E_{c_\varepsilon}|=\omega_d$ and $\rho(E_{c_\varepsilon})=\varepsilon$, and for any $r \in \R$, by \eqref{e56}
\begin{equation}\label{e77}
T_q(E_{c_\varepsilon})^r
\le C_2\omega_d^r\varepsilon^{q'r},
\end{equation}
where $C_2$ is defined in \eqref{c_2}.
Since $E_{c_\varepsilon}$ is contained in a closed ball with radius $\varepsilon^{-1/(d-1)}$ we have by monotonicity of the $p$-capacity and scaling that
\begin{equation}\label{e78}
\cp_p(\overline{E_{c_\varepsilon}})\le \cp_p(\overline{B_{\varepsilon^{-1/(d-1)}}})\le \cp_p(\overline{B_1})\varepsilon^{-(d-p)/(d-1)}.
\end{equation}
Definition \eqref{e59}, \eqref{e77} and \eqref{e78} yield
\begin{equation}\label{e79}
G_{p,q,r,\alpha}(E_{c_\varepsilon})\le C_2 \omega_d^{r-\alpha-\beta}\cp_p(\overline{B_1})\varepsilon^{q'r-\frac{d-p}{d-1}+\beta}.
\end{equation}
The exponent of $\varepsilon$ in the right-hand side of \eqref{e79} is, by the assumption on $r$, strictly positive. This implies (i).

\smallskip

To prove (ii) we recall the following lower bound for the $p$-capacity,  $1<p<d$ proved in \cite[(2.5)]{xiao}
\begin{equation}
\label{xiao}
\cp_p(\overline \Omega) \ge d \omega_d\left( \frac{d(d-p)}{p(d-1)}\right)^{p-1}\left(\frac{P(\Omega)}{d\omega_{d}}\right)^{\frac{d-p}{d-1}}.
\end{equation}
Then by \eqref{xiao} and  \eqref{e54b} we get that
\begin{align*}
G_{p,q,r,\alpha}(\Om)&\ge d \omega_d\left(\frac{P(\Omega)}{d\omega_{d}}\right)^{\frac{d-p}{d-1}}\left( \frac{d(d-p)}{p(d-1)}\right)^{p-1} C_1\, \frac{|\Omega|^{q'r+r-\alpha}}{P(\Omega)^{q'r+\beta}}\nonumber\\&
= (d \omega_d)^{1-\frac{d-p}{d-1}}\left( \frac{d(d-p)}{p(d-1)}\right)^{p-1} C_1\left(\frac{P(\Omega)^{\frac{1}{d-1}}}{|\Omega|^{\frac 1 d}}\right)^{d-p-(d-1)(q'r+\beta)},
\end{align*}
where $C_1$ is defined in \eqref{c1}.
Since $\frac{d-p}{d-1}-q'r-\beta\ge 0$, we get by \eqref{e61a}
\begin{align*}
G_{p,q,r,\alpha}(\Om)&\ge \omega_d^{\frac{1}{d}(p-q'r-\beta)} d^{-\frac{d-p}{d}+\frac{d(p-1)}{d-1}+\frac{d-1}{d}(q'r+\beta)}\nonumber\\&\times\left( \frac{d-p}{p(d-1)}\right)^{p-1} C_1\left(\frac{\textup{diam}(\Om)}{2\rho(\Om)}\right)^{\frac{d-p-(d-1)(q'r+\beta)}{d(d-1)}}.
\end{align*}
By \eqref{e64} we get
\begin{align}\label{e74a}
 G_{p,q,r,\alpha}(\Om)\ge  C_1& \,\left( \frac{p-1}{p(d-1)}\right)^{p-1}d^{\frac{d(p-1)}{d-1}-\frac{2d-p+q'r+\beta}{d}-r+2q'r+2\beta} \nonumber\\&\times(d+q')^r \left(\frac{\textup{diam}(\Om)}{2\rho(\Om)}\right)^{\frac{d-p-(d-1)(q'r+\beta)}{d(d-1)}}G_{p,q,r,\alpha}(B_1).
\end{align}
This implies the assertion under (ii).

\smallskip

To prove the existence of a minimiser we can proceed similarly as in  the proof of Theorem \ref{the_new1}(ii).

\smallskip

Part  (iv)  follows by \eqref{e74a}.

\smallskip

To prove (v) we have by \eqref{e82} that for $r<0$, $T^r_q(\Om)\ge T^r_q(\Om^\sharp)$. Furthermore, $\cp_p(\overline{\Om})\ge \cp_p(\overline{\Om^\sharp})$, and $\frac{T^r_q(\Om^\sharp)\cp_p(\overline{\Om^\sharp)}}{|\Om|^{\alpha}}=G_{p,q,r,\alpha^*}(B_1)$.
\end{proof}

\section{Remarks}\label{sec5}

\begin{remark}\label{rem1}
Two different upper bounds for $p$-capacity are used in the proof of Theorem \ref{the_new1}(i). Inequalities  \eqref{e17} and \eqref{e66} are used in the first and second part respectively. By considering ellipsoids one observes that \eqref{e17} gives better estimates for elongated ellipsoids while \eqref{e66} gives better estimates for flat ellipsoids. Both \eqref{e17} and \eqref{e66} are equalities for a ball.
\end{remark}

\smallskip

\begin{remark}\label{rem2}
In order to show that the supremum in the left-hand side of \eqref{e51} is finite for $1<p\le d-1$ if and only if $q'r\ge p-\beta$ one would have to show that
\[
\cp_p(\overline{E_{a_\varepsilon}}) \ge c(p,d) \varepsilon^{d-p-1}
\]
\end{remark}

\smallskip

In the special case where $p=2$ and $d\ge 3$ we have the following.
\begin{remark}\label{rem3}
If $d\ge 3$, $p=2$, and if $q'r<2-\beta$ then
\begin{equation*}
\sup \{G_{2,q,r,\alpha}(\Om): \Om\, \textup{non-empty, open, convex in}\,\R^d \}=\infty.
\end{equation*}
\end{remark}
\begin{proof}
Consider the ellipsoid $E_{a_\varepsilon}$ of $\R^d$, with semi-axes $a_{\varepsilon}=(1,\varepsilon, \varepsilon,\ldots,\varepsilon)$, then for $\varepsilon\downarrow 0$,
\begin{align}\label{nn}
G_{2,q,r,\alpha}(E_{a_\varepsilon}) \ge \begin{cases}  C(3,q)  \varepsilon^{\frac{q'r+\beta-2}{d}}\big(\log(\varepsilon^{-1})\big)^{-1},&d=3,\\
C(d,q) \varepsilon^{\frac{q'r+\beta-2}{d}} &d>3.
\end{cases}
\end{align}
We have used the lower bound in \eqref{e56}, the definition of $\alpha^*$ in \eqref{e7}, and the formulae for the $2$-capacity of $E_{a_\epsilon}$ as $\varepsilon\downarrow 0$ in \cite[p.260]{IMcK}:
\begin{equation*}%\label{ik}
\cp_2\big(\overline{E_{a_\vps}})\big)=\begin{cases}4\pi\big(\log(\varepsilon^{-1})\big)^{-1}(1+o(1)),&d=3,\ \varepsilon\downarrow0,\\
\ds\frac{2\pi^{d/2}(d-3)}{\Gamma(d/2)}\varepsilon^{d-3}(1+o(1)),&d>3,\ \varepsilon\downarrow0.
\end{cases}
\end{equation*}
If $q'r<2-\beta$, then the right-hand side of \eqref{nn} is unbounded on $(0,1)$.
\end{proof}
The $\log$ factor $\big(\log(\varepsilon^{-1})\big)^{-1}$ for $d=3$ did not play a role in the above as we have the strict inequality $q'r<2$.

\smallskip

It is of course possible to consider the maximisation of $G_{p,q,r,\alpha}(\Om)$ over the collection of open sets in $\R^d$ with finite perimeter or measure. We have the following.
\begin{remark}\label{rem4}
\begin{itemize} \item[]If $d-p>\beta(d-1)$, and if $1<p<d$, then
\begin{equation}\label{x1}
\sup \{G_{p,q,r,\alpha}(\Om): \Om\, \textup{non-empty, open in $\R^d$ with $P(\Om)<\infty$}\}=\infty.
\end{equation}
\item[] If $1<p<d$, then
\begin{equation}\label{x0}
\sup \{G_{p,q,r,\alpha^*}(\Om): \Om\, \textup{non-empty, open in $\R^d$ with $|\Om|<\infty$}\}=\infty.
\end{equation}
\end{itemize}
\end{remark}
\begin{proof}
Let $\Omega_{\varepsilon}$ be the union of the open ball $B_1=\{x\in \R^d:|x|<1\}$, and the ellipsoid
\begin{equation}\label{x2}
E_{b_{\varepsilon}}=\varepsilon^{-3}+\{x\in\R^d: \varepsilon^{-2d}x_d^2+\sum_{j=1}^{d-1}(\varepsilon x_j)^2<1\}.
\end{equation}
Then $B_1\cap E_{b_{\varepsilon}}=\emptyset, \,\varepsilon<\frac12$, and by monotonicity of the $q$-torsion
$T_q(\Om_{\varepsilon})^r\ge T_q(B_1)^r$. By monotonicity of the $p$-capacity,
\begin{align}\label{x3}
\cp_p(\overline{\Om_{\varepsilon}})&\ge \cp_p(\overline{E_{b_{\varepsilon}}})\ge d \omega_d\left(\frac{P(E_{b_{\varepsilon}})}{d\omega_{d}}\right)^{\frac{d-p}{d-1}}\left( \frac{d(d-p)}{p(d-1)}\right)^{p-1}.
\end{align}
We also have that $|\Om_{\varepsilon}|=\omega_d(1+\varepsilon)\le 3\omega_d/2$, and $P(\Om_{\varepsilon})=P(B_1)+P(E_{b_{\varepsilon}})$.
Putting this together yields
\begin{equation}\label{x4}
G_{p,q,r,\alpha}(\Om_{\varepsilon})\ge k_d\frac{T_q(B_1)^r(P(E_{b_{\varepsilon}}))^{\frac{d-p}{d-1}}}{(3\omega_d/2)^{\alpha}(P(B_1)+P(E_{b_{\varepsilon}}))^{\beta}},
\end{equation}
where $k_d$ can be read-off from \eqref{x3}. By the second inequality in \eqref{e61},
\begin{equation}\label{x5}
P(E_{\varepsilon})\ge \frac{|E_{\varepsilon}|}{a_d}=\omega_d\varepsilon^{1-d}.
\end{equation}
Since $d-p>\beta(d-1)$, and $P(E_{\varepsilon})\rightarrow\infty$ as $\varepsilon\downarrow 0$ we conclude that the right-hand side of \eqref{x4} tends to $\infty$.
This implies \eqref{x1}.

To prove \eqref{x0} we follow the lines above with $\beta=0$.
\end{proof}
\begin{remark}\label{rem4}
The estimates in the proof of Theorem \ref{the_new1} (or Theorem \ref{the_min}) are too crude to discuss the existence of a maximiser (minimiser) if inequality \eqref{e52} is not strict. However, Theorem \ref{the_new1}(iv) gives a partial result in that case.
\end{remark}

\section*{Acknowledgements}
The second author was supported by the Project MUR PRIN-PNRR 2022:  "Linear and Nonlinear PDE’S: New directions and Applications", P2022YFA and by
GNAMPA group of INdAM.

\section*{Conflicts of interest and data availability statement}
The authors declare that there is no conflict of interest. Data sharing not applicable to this article as no datasets were generated or analyzed during the current study.

\end{document}